\documentclass{birkjour}
 \newtheorem{thm}{Theorem}[section]
 \newtheorem{cor}[thm]{Corollary}
 \newtheorem{lem}[thm]{Lemma}
 \newtheorem{prop}[thm]{Proposition}
 \theoremstyle{definition}
 \newtheorem{defn}[thm]{Definition}
 \newtheorem{ex}[thm]{Example}
 \theoremstyle{remark}
 \newtheorem{rem}[thm]{Remark}
 \numberwithin{equation}{section}
\usepackage{amsmath, amsthm, amscd, amsfonts, amssymb, graphicx, color,float,pgf,tikz}
\usepackage{amssymb,fontenc}
\usepackage{color}
\begin{document}

%
%
%
%
%
%
%
%
%

\title[Characterization of (weak) phase  retrieval dual frames]
 {Characterization of (weak) phase  retrieval dual frames}

\author[F. Arabyani-Neyshaburi]{Fahimeh  Arabyani-Neyshaburi$^{1}$}
\email{f.arabiani@um.ac.ir, fahimeh.arabyani@gmail.com}

\author[A. Arefijamaal] {Ali Akbar Arefijamaal$^{2}$} 
\email{arefijamaal@hsu.ac.ir;arefijamaal@gmail.com}

\author[R.Kamyabi-Gol] {Rajab Ali  Kamyabi-Gol$^{3,*}$} \let\thefootnote\relax\footnotetext{$^{*}$ Corresponding author.}
\email{kamyabi@um.ac.ir}

\maketitle
\vspace{-1.5cm}
\begin{center}
$^{1}$ Department of  Mathematical sciences, Ferdowsi University of Mashhad, Mashhad, Iran.\\
$^{2}$ Department of Mathematics and Computer Sciences, Hakim Sabzevari University, Sabzevar, Iran.\\
 $^{3}$ Department of Mathematical sciences, Faculty of Math, Ferdowsi University of Mashhad and Center of Excellence in Analysis on Algebraic Structures (CEAAS), Mashhad, Iran. Email: \email{kamyabi@um.ac.ir} \\
\end{center}



\vspace{1.7cm}
\begin{abstract}
Recovering a signal up to a unimodular constant from the magnitudes of linear measurements has been popular and well studied in recent years.
However,  numerous unsolved problems regarding phase retrieval still exist. Given a phase retrieval frame, may the family of phase retrieval dual frames be classified? And  is such a family  dense in the set of dual frames? Can we present the equivalent conditions for a family of vectors to do weak phase retrieval in complex Hilbert space case? What is the connection between phase, weak phase and norm retrieval?
In this context, we aim to deal with these open problems concerning   phase  retrieval dual frames,   weak phase  retrieval frames,
and specially  investigate equivalent conditions for identifying these features.
We provide some characterizations  of   alternate dual frames of a phase retrieval frame which yield phase retrieval in finite dimensional Hilbert spaces. Moreover, for some classes of frames,  we show that the family of  phase retrieval dual frames  is open and    dense in the set of  dual frames.
 Then, we study weak phase retrieval  problem. Among other things, we     obtain  some  equivalent conditions on a family of vectors to do    phase  retrieval in terms of weak phase retrieval.
\end{abstract}

\maketitle

%



\maketitle

\section{Introduction and Preliminaries}

\smallskip
\goodbreak
Signal reconstruction without using phase is a longstanding conjecture, specially with regard to speech recognition systems which  was first introduced  from mathematical  point of view by  R. Balan,  P. Casazza and D. Edidin in \cite{Casazza}, and has been a very popular topic in recent years  due to so many applications of  $X$-ray, electron microscopy,  optics, image processing and much more \cite{Balan, Efriam,Van, Wang}.
To present the problem in a more precise approach,  we first  briefly state some basic definitions and preliminaries  in relevant areas. Then  we propound a number of problems to address phase retrieval frames in some  new aspects.

Throughout this paper, we suppose that $\mathcal{H}$ denotes a separable Hilbert space,  $\mathcal{H}_{n}$ denote an $n$-dimensional  real or complex Hilbert space  and  we use $\mathbb{R}^{n}$ and $\mathbb{C}^{n}$ whenever it is necessary  to differentiate between the two.  Also, we consider the notations; $I_{m}=\{1, 2, ..., m\}$ and     $\{\delta_{i}\}_{i\in I_{n}}$ as the standard orthonormal basis of $\mathcal{H}_{n}$.

A family of vectors
$\Phi:=\{\phi_{i}\}_{i\in I}$ in   $\mathcal{H}$ is called a  \textit{frame}  if there exist the constants
 $0<A_{\Phi}\leq B_{\Phi}<\infty$ such that
\begin{eqnarray}\label{Def frame}
A_{\Phi}\|f\|^{2}\leq \sum_{i\in I}\vert f,\phi_{i}\rangle\vert^{2}\leq
B_{\Phi}\|f\|^{2},\qquad (f\in \mathcal{H}).
\end{eqnarray}
The constants $A_{\Phi}$ and $B_{\Phi}$ are called \textit{frame bounds}. The sequence  $\{\phi_{i}\}_{i\in I}$ is said to be  a \textit{Bessel sequence} whenever the right hand side of  (\ref{Def frame})  holds.
 A frame
 $\{\phi_{i}\}_{i\in I}$ is called   $A$-\textit{tight frame} if $A=A_{\Phi}=B_{\Phi}$, and in the case of $A_{\Phi}=B_{\Phi}=1$ it is called a  \textit{Parseval frame}.
Given a
frame $\Phi=\{\phi_{i}\}_{i\in I}$, its Grammian matrix formed by the inner product of the frame vectors is as $G_{\Phi}=[\langle \phi_{i},\phi_{j}\rangle]_{i,j}$.
The \textit{frame operator} is defined by
$S_{\Phi}f=\sum_{i\in I}\langle f,\phi_{i}\rangle \phi_{i}$.
 It is a bounded, invertible, and self-adjoint
operator \cite{Chr08}.
Also, the \textit{synthesis operator} $T_{\Phi}: l^{2}\rightarrow \mathcal{H}$ is defined  by $T_{\Phi}\lbrace c_{i}\rbrace= \sum_{i\in I} c_{i}\phi_{i}$. The frame operator can be written as $S_{\Phi}= T_{\phi}T_{\phi}^{*}$ where $T_{\Phi}^{*}: \mathcal{H}\rightarrow l^{2}$, the adjoint of $T$, given by $T_{\Phi}^{*}f= \lbrace \langle f,\phi_{i}\rangle\rbrace_{i\in I}$  is called the \textit{analysis operator}.
The family $\{S_{\Phi}^{-1}f_{i}\}_{i\in I}$ is also a frame for $\mathcal{H}$,   called the \textit{canonical dual} frame.
In general, a frame $\{\psi_i\}_{i\in I}\subseteq\mathcal{H}$ is called an \textit{alternate dual} or simply a \textit{dual}
 for   $\{\phi_{i}\}_{i\in I}$ if
$
f=\sum_{i\in I}\langle f,\psi_{i}\rangle \phi_i ,
$
for $f\in
\mathcal{H}$.
All frames have at least a dual, the canonical dual, and redundant frames have an  infinite number of  alternate  dual frames. We denote the excess of a frame $\Phi$ by $E(\Phi)$.
It is  known that  every dual frame is of the form  $\{S_{\Phi}^{-1}\phi_{i}+u_{i}\}_{i\in I}$, where   $\{u_{i}\}_{i\in I}$ is a Bessel sequence that  satisfies
$
\sum_{i\in I}\langle f,u_{i}\rangle \phi_{i}=0,
$
for all $f\in
\mathcal{H}$. Also recall that, two frames $\Phi$ and $\Psi$ are equivalent if there exists an invertible operator $U$ on $\mathcal{H}$ so that $\Psi=U\Phi$.
 See \cite{Chr08, H} for more detailed information on frame theory and \cite {Aa2018, AAS, Bol98,   lopez, Kovaevi, Saliha} for the importance of duality principle.

Consider a frame  $\{\phi_{i}\}_{i\in I_{m}}$ in a Hilbert space $\mathcal{H}_{n}$.   A finite set of indices $\sigma\subset I_{m}$ satisfies the minimal redundancy condition (MRC)  whenever $\{\phi_{i}\}_{i\in \sigma^{c}}$ remains to be a  frame for $\mathcal{H}_{n}$. Furthermore,  we say $\Phi$ satisfies MRC for $r$-erasures if every subset $\sigma\subset I_{m}$, $\vert \sigma\vert=r$ satisfies MRC for $\Phi$.  The spark of a matrix is the size of the smallest linearly dependent subset of the columns and the spark of a family $\{\phi_{i}\}_{i\in I_{m}}$  is defined as the spark of its synthesis matrix $\Phi$. Also, the family $\{\phi_{i}\}_{i\in I_{m}}$, $m\geq n$ is said   full spark if it has the spark $n+1$. It is shown   that if $m\geq n$, then the set of full spark frames is open and dense in the set of all frames \cite{Alexeev}.

Suppose now the nonlinear mapping
\begin{eqnarray}\label{Phase map}
\mathbb{M}_{\Phi}: \mathcal{H}\rightarrow l^{2}(I), \qquad \mathbb{M}_{\Phi}(f)=\lbrace \vert\langle f,\phi_{i}\rangle\vert^{2}\rbrace_{i\in I}
\end{eqnarray}
obtained by taking the absolute value element wise of the analysis operator. Let us denote by $\mathbb{H}=\mathcal{H}/\sim$   considered by identifying two vectors which are different in a phase factor, i.e., $f \sim g$ whenever there exists a scalar $\theta$ with $\vert \theta \vert=1$  so that $g=\theta f$.
Obviously in a real Hilbert space  we have  $\mathbb{H}=\mathcal{H}/\{1,-1\}$ and in the complex case $\mathbb{H}=\mathcal{H}/\mathbb{T}$, where $\mathbb{T}$ is the complex unit circle. So, the mapping $\mathbb{M}_{\Phi}$ can be extended to $\mathbb{H}$ as $\mathbb{M}_{\Phi}(\hat{f})=\lbrace \vert\langle f,\phi_{i}\rangle\vert^{2}\rbrace_{i\in I}$, $f\in \hat{f}=\{g\in \mathcal{H}: g \sim f\}$.
The injectivity of the nonlinear mapping $\mathbb{M}_{\Phi}$   leads to the reconstruction of every signal  in $\mathcal{H}$ up to a constant phase factor from the modula   of its frame coefficients.
In \cite{Casazza}, the authors investigated the injectivity of   $\mathbb{M}_{\Phi}$ in finite dimensional real Hilbert spaces and moreover, it was proven that $4n-2$ measurements suffice for the  injectivity in $n$-dimensional complex Hilbert spaces. Indeed, the injectivity of the mapping $\mathbb{M}_{\Phi}$ is equivalent to the following definition:

\begin{defn}\label{1}
A family of vectors $\Phi=\{\phi_{i}\}_{i\in I}$ in $\mathcal{H}$ does phase  retrieval if whenever
$f,g\in \mathcal{H}$ satisfy
\begin{eqnarray}\label{121}
\vert \langle f,\phi_{i}\rangle\vert=\vert \langle g,\phi_{i}\rangle\vert, \qquad
 (i\in I)
\end{eqnarray}
then there exists a scalar $\theta$ with $\vert \theta\vert =1$ so that $f= \theta g$.
\end{defn}
If for every
$f,g\in \mathcal{H}$, which satisfy $(\ref{121})$ we get
$\Vert f\Vert =\Vert g\Vert$, then it said $\Phi$ to do  norm retrieval.
Clearly, if $\Phi$ does phase (norm)  retrieval, then so does $\alpha_{i}f_{i}$ for every $0<\alpha_{i}$, $i\in I$. Also, tight frames do norm retrieval.
Moreover,  phase retrieval implies norm retrieval, but the converse fails. For example every  orthonormal basis does norm retrieval, but fails at  phase retrieval.

The following result states that for two equivalent frames $\Phi$ and $\psi$, the injectivity of $\mathbb{M}_{\Phi}$ and  $\mathbb{M}^{\Psi}$ are the same.

\begin{thm}\label{1.3}\cite{Casazza}
A family $\Phi=\{\phi_{i}\}_{i\in I}$  in $\mathcal{H}$ does phase retrieval if and only if $\{U\phi_{i}\}_{i\in I}$ does phase retrieval for every invertible operator $U$ on $\mathcal{H}$.\end{thm}
Applying the above theorem, shows that a frame does phase retrieval if and only if  its canonical dual   does phase retrieval.
\begin{defn}\label{2}\cite{Casazza}
A family of vectors $\Phi=\{\phi_{i}\}_{i\in I}$ in  $\mathcal{H}$ has the complement property if for every $\sigma \subset I$ either $\overline{span}\{\phi_{i}\}_{i\in \sigma}=\mathcal{H}$ or $\overline{span}\{\phi_{i}\}_{i\in \sigma^{c}}=\mathcal{H}$.
\end{defn}
  As we stated, a fundamental  classification of frames which do phase retrieval was presented for finite dimensional real case in \cite{Casazza} and then for infinite dimensional case  in \cite{cahill}  as follows:
 \begin{thm}\label{thm1}
A family $\Phi=\{\phi_{i}\}_{i\in I}$  in a real Hilbert space $\mathcal{H}$ does phase retrieval if and only if it has the complement property.
\end{thm}
As an  immediate result of the above theorem,  every phase retrieval frame in real Hilbert space $\mathcal{H}$ is satisfied in MRC for $(n-1)$-erasures.

\begin{prop}\cite{Casazza}
If  $\Phi=\{\phi_{i}\}_{i\in I_{m}}$  does phase retrieval in $\mathbb{R}^{n}$, then $m\geq 2n-1$. If $m\geq 2n-1$ and $\Phi$ is full spark then $\mathcal{\phi}$  does phase retrieval. Moreover, $\{\phi_{i}\}_{i\in I_{2n-1}}$  does phase retrieval if and only if $\Phi$ is full spark.\end{prop}

Two vectors $x=\{x_{i}\}_{i\in I_{n}}$ and $y=\{y_{i}\}_{i\in I_{n}}$ in $\mathcal{H}_{n}$  weakly have the same phase if there is
a $\vert \alpha\vert=1$ so that $phase (x_{i})=\alpha phase (y_{i})$, for all $i\in I_{n}$ which $x_{i}\neq 0 \neq y_{i}$.

\begin{defn}A family $\Phi=\{\phi_{i}\}_{i\in I_{m}}$  in $\mathcal{H}_{n}$  does weak phase retrieval if for any $x, y \in \mathcal{H}_{n}$ with $\vert \langle x,\phi_{i}\rangle\vert=\vert \langle y,\phi_{i}\rangle\vert$ for all $i\in I_{m}$, then  $x$ and $y$ weakly have the same phase.\end{defn}
It is shown that if $\Phi=\{\phi_{i}\}_{i\in I_{m}}$   does weak phase retrieval in $\mathbb{R}^{n}$, then $m\geq 2n-2$ \cite{Casazza22}. Also, clearly phase retrieval implies weak phase retrieval property, although the converse does not hold in general. As a simple example $\{(1,1), (1,-1)\}$ does weak phase retrieval for $\mathbb{R}^{2}$,  see\cite{Akrami}, but clearly does not phase retrieval.

Now, we state the concept of a generic set; A subset $\Omega\subseteq \mathbb{R}^{n}$ is called generic whenever there exists a nonzero polynomial $p(x_{1},...,x_{n})$ so that
\begin{equation*}
\Omega^{c}\subseteq \{(x_{1},...,x_{n})\in \mathbb{R}^{n} : p(x_{1},...,x_{n})=0\}.
\end{equation*}
 It is known that   generic sets are open, dense and full measure. Furthermore, a generic set in $\mathbb{C}^{n}$ is defined as a generic set in $\mathbb{R}^{2n}$.

 Applying  the linear operator introduced in \cite{Balan1}  gives an equivalent condition for injectivity of $\mathbb{M}_{\Phi}$, defined by $(\ref{Phase map})$. More precisely, let $\{\phi_{i}\}_{i\in I_{m}}$ be a family of vectors in $\mathbb{C}^{n}$ and $H_{n\times n}$ denotes the space of all $n\times n$ Hermitian matrices. Consider the operator $\Lambda_{\Phi}:H_{n\times n}\rightarrow \mathbb{R}^{m}$ by $\Lambda_{\Phi}(A)=\{\langle A, \phi_{i}\otimes \phi_{i}\rangle\}_{i\in I_{m}}$, in which $\phi_{i}\otimes \phi_{i}$ is the rank one projection onto $span\{\phi_{i}\}$. It is easily proven that $\Lambda_{\Phi}(f\otimes f)=\mathbb{M}_{\Phi}(f)$ and this equality has a key role  for characterizing  the injectivity of $\mathbb{M}_{\Phi}$, \cite{Balan1, thesis}, See also \cite{hasankhani}.
 So, we get the following equivalent conditions for a frame to do phase  retrieval.
 \begin{cor}
Let $\Phi=\{\mathcal{\phi}_{i}\}_{i\in I_{m}}$ be a frame in $\mathcal{H}_{n}$, then the followings are equivalent;
\begin{itemize}
\item[(i)]
$\Phi$ does phase retrieval.
\item[(ii)] $\mathbb{M}_{\Phi}$ is injective.
\item[(iii)] $\Lambda_{\Phi}\vert_{B_{1}}$ is injective, where $B_{1}$ denotes rank one matrices.
\item[(iv)] There exists no rank 2 matrix in the null space of $\Lambda_{\Phi}$.
\end{itemize}
\end{cor}
\begin{proof}
$(i) \Leftrightarrow  (ii)$ is obvious by the definition. For $(ii) \Leftrightarrow  (iv)$ we
just note that, due to the completeness of $\Phi$ in $\mathcal{H}_{n}$, the rank one matrices  $B_{1}=\{f\otimes f : \quad f\in \mathcal{H}_{n}\}$ can not be in the null space of $\Lambda_{\Phi}$. Indeed, $\Lambda_{\Phi}(f\otimes f)=\mathbb{M}_{\Phi}(f)=0$ implies that $f\perp \phi_{i}$, for all $i\in I_{m}$ and so $f=0$. So, as a result of  Lemma 5.5 of  \cite{Balan1}, the injectivity of $\mathbb{M}_{\Phi}$ is equivalent to the statement that; there is  no rank 2 matrix in the null space of $\Lambda_{\Phi}$. On the other hand, the injectivity of $\mathbb{M}_{\Phi}$ along with the fact that the mapping $\gamma: \mathcal{H}_{n}\rightarrow B_{1}$, $\gamma(f)=f\otimes f$ is invertible, deduce that the linear operator $\Lambda_{\Phi}\vert_{B_{1}}$ is also injective. Moreover, for a left inverse $\mathbb{L}_{\Phi}$ of $\mathbb{M}_{\Phi}$, the mapping $\gamma\mathbb{L}_{\Phi}$ is a left inverse for $\Lambda_{\Phi}\vert_{B_{1}}$. Conversely,  for a left inverse $\Gamma_{\Phi}$ of $\Lambda_{\Phi}\vert_{B_{1}}$, the mapping $\gamma^{-1}\Gamma_{\Phi}$ is a left inverse for $\mathbb{M}_{\Phi}$, this completes the proof of $(ii) \Leftrightarrow  (iii)$.
\end{proof}
This  paper is organized as follows: Section 2  is devoted to characterizing phase retrieval dual frames and full spark dual frames. In this section, for some classes of frames,  we show that the family of all ( full spark) phase retrieval dual frames  is open and    dense in the set of all dual frames. Moreover, we present several examples in this regard. In Section 3,
 we survey weak phase retrieval problem  and investigate some equivalent conditions for identifying  phase  and weak phase retrieval frames. We also, obtain
a sufficient conditions on a family of weak phase retrieval frames to constitute a frame and its canonical dual yields weak phase retrieval, as well.
\section{Phase  Retrieval Dual Frames}
In this section, we address the problem that, given a phase retrieval frame $\Phi$, characterize phase   retrieval dual frames of $\Phi$ in finite dimensional Hilbert spaces. For some classes of frames we show that  phase   retrieval dual frames of a given frame are dense in the set of all dual frames.  For a frame $\Phi$ we  denote  the set of all its dual frames of $\Phi$ by $D_{\Phi}$ and the subset of phase retrieval dual frames is denoted  by $PD_{\Phi}$.
We first investigate the relationship between  phase retrieval  duals of equivalent frames, which is a very useful tool for the main results of this section.
\begin{lem}\label{equlem}
Suppose that $\Phi=\{\phi_{i}\}_{i\in I_{m}}$ is a frame for $\mathcal{H}_{n}$  and $\mathcal{T}$ is an invertible operator on $\mathcal{H}_{n}$.  Then,
\begin{itemize}
\item[(i)]
$D_{\mathcal{T}\Phi} = (\mathcal{T}^{*})^{-1}  D_{\Phi}$.
\item[(ii)] $PD_{\mathcal{T}\Phi} = (\mathcal{T}^{*})^{-1} PD_{\Phi}$.
\end{itemize}
\end{lem}
\begin{proof}
It is known that  $\mathcal{T}\Phi$ is a frame with $S_{\mathcal{T}\Phi}= \mathcal{T}S_{\Phi}\mathcal{T}^{*}$ \cite{Chr08},  and so  a simple computation assures that $D_{\mathcal{T}\Phi} = \{(\mathcal{T}^{*})^{-1}G : G\in D_{\Phi}\}=  (\mathcal{T}^{*})^{-1}  D_{\Phi}$.
Moreover, $(ii)$ is given as  an immediate result of $(i)$ along with Theorem \ref{1.3}.
\end{proof}
\begin{rem}\label{2.212} If $\Phi=\{\phi_{i}\}_{i\in I_{m}}$ is a frame for $\mathcal{H}_{n}$ so that  $E(\Phi)=k$, which  $E(\Phi)$ denotes the excess of $\Phi$, then there exist  $i_{1}, ..., i_{k}$ so that $\Phi \setminus \{\phi_{i_{j}}\}_{j=1}^{k}$ constitutes a Riesz basis for $\mathcal{H}_{n}$. Therefore, applying the above notations and without loss of generality,  we may consider
$
\Phi =\{\phi_{i}\}_{i\in I_{n}}\cup \{\phi_{i}\}_{i=n+1}^{m}$, where $\{\phi_{i}\}_{i\in I_{n}}$ is a  Riesz basis  for $\mathcal{H}_{n}$. In this point of view, $\Phi$ is indeed equivalent to a form as $\{\delta_{i}\}_{i\in I_{n}}\cup \{\tilde{\phi}_{i}\}_{i=n+1}^{m}$ with redundant elements $\{\tilde{\phi}_{i}\}_{i=n+1}^{m}$.
\end{rem}
In the next theorem we identify the set of dual frames of a frame $\{\phi_{i}\}_{i\in I_{2n-1}}$ in $n$-dimensional real space and show that $PD_{\phi}$  is an open and  dense subset  in $D_{\phi}$.
\begin{thm}\label{densemetric}
Let $\phi=\{\phi_{i}\}_{i\in I_{2n-1}}$ be a  phase retrieval frame in $\mathbb{R}^{n}$. Then   $PD_{\phi}$  is an open and  dense subset  in $D_{\phi}$.
\end{thm}
\begin{proof}
First let $\{\phi_{i}\}_{i\in I_{n}}$ be the standard orthonormal  basis of $\mathbb{R}^{n}$ and  $\phi_{j}=\sum_{i\in I_{n}} a_{i}^{j} \phi_{i}$, $j= n+1, ... ,2n-1$, for some  non-zero coefficients $\{a_{i}^{j}\}_{i\in I_{n}}$.
 Due to the fact that
\begin{eqnarray*}
D_{\Phi} = \left\{\{S_{\Phi}^{-1}\phi_{i}+u_{i}\}_{i\in I_{2n-1}} : \sum_{i\in I_{2n-1}}\langle f,\phi_{i}\rangle u_{i}=0, \textit{for all } f\in \mathbb{R}^{n}\right\}
\end{eqnarray*}
by putting $f=\phi_{j}$, $j=1, ..., 2n-1$ we observe that $(2n-1) \times n$ matrix $[u_{1} \vert  u_{2} \vert ...  \vert u_{2n-1}]^{T}$ is  in the null space of $G_{\Phi}^{T}$, the transposed Gram matrix of $\Phi$. The fact that,   $dim (null G_{\Phi} )=n-1$;
assures that just $n-1$ vectors of $\{u_{i}\}_{i\in I_{2n-1}}$ can  be  independent. More precisely,
\begin{eqnarray*}
 u_{1}+\sum_{i=n+1}^{2n-1}\langle \phi_{1},\phi_{i}\rangle u_{i}=0,\quad  ... \quad u_{n}+\sum_{i=n+1}^{2n-1}\langle \phi_{n},\phi_{i}\rangle u_{i}=0.
\end{eqnarray*}
 So, by choosing $\{u_{j}\}_{j=n+1}^{2n-1}$ we get
\begin{equation}\label{constuctu}
u_{i}=  - \sum_{j=n+1}^{2n-1}a_{i}^{j} u_{j}, \quad (i\in I_{n}),
\end{equation}
in which $a_{i}^{j}=\langle \phi_{i},\phi_{j}\rangle$, for $n+1\leq j\leq 2n-1$.
 Hence,  every dual frame   is uniquely constructed by  the following vector
\begin{eqnarray}\label{uniqueg}
\mathcal{U}  = \left[u_{n+1},..., u_{2n-1}\right]
%
\in \mathbb{R}^{n(n-1)}.
\end{eqnarray}
Considering $D_{\Phi}$ as a metric space  by  $d(G, H)=\sum_{i\in I_{2n-1}}\Vert  g_{i}-h_{i} \Vert$, we
define the mapping $\xi:D_{\Phi}\rightarrow  \mathbb{R}^{n(n-1)}$  by $\xi(G) = \mathcal{U}_{g}$, where   $\mathcal{U}_{g}\in \mathbb{R}^{n(n-1)}$ is  the unique sequence associated to $G\in D_{\Phi}$ as in $(\ref{uniqueg})$.
Then, clearly
$\xi$ is well-defined  and injective. Also, take $\mathcal{A} \in \mathbb{R}^{n(n-1)}$, then put $u_{n+1}=\{\mathcal{A}_{1}, ... , \mathcal{A}_{n}\}$,..., $u_{2n-1}=\{\mathcal{A}_{n^{2}-2n+1}, ... , \mathcal{A}_{n(n-1)}\}$ and construct $u_{i}$, $i\in I_{n}$ by $(\ref{constuctu})$. Thus we get  $\{S_{\Phi}^{-1}\phi_{i}+u_{i}\}_{i\in I_{2n-1}} \in D_{\Phi}$, i.e., $\xi$ is a bijective map.
Moreover,  $\xi$ is a  Lipschitz function.
Let $G=\{S_{\Phi}^{-1}\phi_{i}+u_{i}\}_{i\in I_{2n-1}}$ and $H=\{S_{\Phi}^{-1}\phi_{i}+v_{i}\}_{i\in I_{2n-1}}$ be dual frames of $\Phi$ with the associated $\mathcal{U}_{g}$ and $\mathcal{U}_{h}$ obtained  as in $(\ref{uniqueg})$, respectively.  Then,
\begin{eqnarray*}
\Vert\xi(G)-\xi(H)\Vert = \Vert\mathcal{U}_{g}-\mathcal{U}_{h}\Vert = \sum_{i=n+1}^{2n-1} \Vert u_{i}-v_{i}\Vert
\leq \sum_{i\in I_{2n-1}} \Vert u_{i}-v_{i}\Vert =d(G, H).
 \end{eqnarray*}
 What is more, $\xi$ is a bi-Lipschitz function, but we will not need this fact.
Now, we note that   the only cases in which a dual frame $\{g_{i}\}_{i\in I_{2n-1}}$ fails to do phase retrieval is associated to
$ det [g_{i_{1}}\vert \quad ... \quad  \vert  g_{i_{n}}]  =0$
for some index set $\{i_{j}\}_{j=1}^{n}\subset I_{2n-1}$, by Theorem \ref{thm1}.
Multiplying these   equations yields an $n(n-1)$-variable polynomial in terms of $u_{n+1}$, ..., $u_{2n-1}$,   denoted by  $P_{\mathcal{\phi}}$. Therefore,
 \begin{eqnarray}\label{open}
  PD_{\Phi}=\xi^{-1}\{\mathcal{U}\in \mathbb{R}^{n(n-1)}: P_{\mathcal{\phi}}(\mathcal{U})\neq 0\},
 \end{eqnarray}
 that means $PD_{\Phi}$ is open in $D_{\Phi}$.
To show $PD_{\Phi}$ is, moreover, dense in $D_{\phi}$, let   $\epsilon>0$ and $G=\{g_{i}\}_{i\in I_{2n-1}} = \{S_{\Phi}^{-1}\Phi_{i}+u_{i}\}_{i\in I_{2n-1}}$ be a dual frame in $D_{\Phi}\setminus PD_{\Phi}$. Then $G$ is dependent just  in vectors $\{u_{i}\}_{i=n+1}^{2n-1}$ such that $P_{\mathcal{\phi}}(\{u_{i}\}_{i=n+1}^{2n-1})=0$. Take a sequence $ \{\{v_{i}^{k}\}_{i=n+1}^{2n-1}\}_{k\in \mathbb{N}}$  out of the roots of the polynomial $P_{\mathcal{\phi}}$ in $\mathbb{R}^{n(n-1)}$ so that
 $\lim_{k\rightarrow \infty} \{v_{i}^{k}\}_{i=n+1}^{2n-1}=\{u_{i}\}_{i=n+1}^{2n-1}$. Also, put $v_{i}^{k}=-\sum_{j=n+1}^{2n-1}  a_{i}^{j} v_{j}^{k}$, for all $1\leq i\leq n$ and $k\in \mathbb{N}$.
Then, $\{v_{i}^{k}\}_{i\in  I_{2n-1}}$ satisfies $(\ref{constuctu})$ and so
the sequence $ \{h_{i}^{k}\}_{k\in \mathbb{N}, i\in I_{2n-1}}$  associated to $ \{\{v_{i}^{k}\}_{i=n+1}^{2n-1}\}_{k\in \mathbb{N}}$ defined by
$h_{i}^{k}= S_{\Phi}^{-1}\Phi_{i}+v_{i}^{k}$, $i\in I_{2n-1}$
constitutes a dual frame of $\Phi$, for all $k\in \mathbb{N}$.
Furthermore, there exists $k_{0}\in \mathbb{N}$ so that for $k \geq k_{0}$
\begin{eqnarray*}
d\left(\{v_{i}^{k}\}_{i=n+1}^{2n-1}\},\{u_{i}\}_{i=n+1}^{2n-1}\}\right) = \sum_{i=n+1}^{2n-1}\Vert v_{i}^{k}-u_{i}\Vert< \epsilon.
\end{eqnarray*}
Hence,  we can write
\begin{eqnarray*}
d(\{h_{i}^{k}\}_{k\in \mathbb{N}, i\in I_{2n-1}},\{g_{i}\}_{i\in I_{2n-1}}) &=& \sum_{i\in I_{2n-1}} \Vert v_{i}^{k}-u_{i}\Vert\\
&=& \sum_{i\in I_{n}} \Vert v_{i}^{k}-u_{i}\Vert +\sum_{i=n+1}^{2n-1} \Vert v_{i}^{k}-u_{i}\Vert\\
&=&  \sum_{i\in I_{n}} \Vert -\sum_{j=n+1}^{2n-1}a_{i}^{j}(v_{j}^{k}-u_{j})\Vert +\sum_{i=n+1}^{2n-1} \Vert v_{i}^{k}-u_{i}\Vert\\
&\leq&  \sum_{i\in I_{n}}\sum_{j=n+1}^{2n-1}\vert a_{i}^{j}\vert  \Vert  v_{j}^{k}-u_{j}\Vert +\sum_{i=n+1}^{2n-1} \Vert v_{i}^{k}-u_{i}\Vert\\
&\leq&(\alpha+1)  \sum_{j=n+1}^{2n-1} \Vert  v_{j}^{k}-u_{j}\Vert
< \epsilon(\alpha+1),
\end{eqnarray*}
for $k \geq k_{0}$, i.e., $PD_{\Phi}$ is dense in $D_{\Phi}$.

Now, applying  Lemma \ref{equlem} and Remark \ref{2.212} gives   the result  in general case.
In fact,   every frame $\Psi = \{\psi_{i}\}_{i\in I_{2n-1}}$ is equivalent to a frame in the form of $\Phi=
\{\delta_{i}\}_{i\in I_{n}}\cup \{\phi _{i}\}_{i=n+1}^{2n-1}$, as we discussed in Remark \ref{2.212}. So, there exists an invertible operator $\mathcal{T}$ on $\mathbb{R}^{n}$ so that $\Psi=\mathcal{T}\Phi$, consequently by Lemma \ref{equlem} we get
\begin{eqnarray*}
\overline{PD_{\Psi}}= \overline{PD_{\mathcal{T}\Phi}} = \overline{(\mathcal{T}^{*})^{-1} PD_{\Phi}}\supseteq  (\mathcal{T}^{*})^{-1}\overline{  PD_{\Phi}}=(\mathcal{T}^{*})^{-1}D_{\Phi}=D_{\Psi},
\end{eqnarray*}
i.e., $\overline{PD_{\Psi}}=D_{\Psi}$. And using $(\ref{open})$
\begin{eqnarray*}
PD_{\Psi}= (\mathcal{T}^{*})^{-1} PD_{\Phi} =(\mathcal{T}^{*})^{-1}\xi^{-1}\{\mathcal{U}\in \mathbb{R}^{n(n-1)}: P_{\mathcal{\phi}}(\mathcal{U})\neq 0\},
\end{eqnarray*}
i.e., $PD_{\Psi}$ is open in $D_{\Psi}$.
 This completes the proof.
\end{proof}
An analogous approach to the proof of  Theorem \ref{densemetric} deduces  that  for  any full spark frame  $\Phi$ in $\mathbb{R}^{n}$ with $E(\Phi)=1$, the set of all full spark dual frames is embedded into a generic set in $\mathbb{R}^{n}$.
\begin{cor}
Let $\Phi=\{\phi_{i}\}_{i=1}^{n+1}$ be a full spark frame for $\mathbb{R}^{n}$. Then   the set of all full spark dual frames of $\Phi $   is an open and dense subset of   $D_{\Phi}$ and is embedded into a generic set in $\mathbb{R}^{n}$.
\end{cor}
\begin{proof}
Applying  Remark \ref{2.212},  a full spark frame $\Phi=\{\phi_{i}\}_{i=1}^{n+1}$ of $\mathbb{R}^{n}$ is equivalent to $\tilde{\Phi}:=\{\delta_{i}\}_{i=1}^{n}\cup \{\sum_{i=1}^{n}\alpha_{i}\delta_{i}\}$ for non-zero scalars $\alpha_{i}$, $i\in I_{n}$, and in this case  the frame operator is as follows:
\begin{equation*}
S_{\tilde{\Phi}} = \left[
 \begin{array}{ccc}

1+\alpha_{1}^{2} \quad \alpha_{1}\alpha_{2} \quad ... \quad  \alpha_{1}\alpha_{n}\\
\\
\alpha_{1}\alpha_{2} \quad 1+\alpha_{2}^{2} \quad ... \quad \alpha_{2}\alpha_{n}\\
.\\
.\\
\alpha_{1}\alpha_{n}\quad \alpha_{2}\alpha_{n} \quad ... \quad 1+\alpha_{n}^{2} \\
\end{array} \right]
\end{equation*}
Also,
 every dual frame of $\tilde{\Phi}$  is in the form of $\{S_{\tilde{\Phi}}^{-1}\tilde{\phi}_{i}+u_{i}\}_{i=1}^{n+1}$  so that  $T_{u}T^{*}_{\tilde{\Phi}}=0$. By taking $u_{n+1}=[x_{1}, ... , x_{n}]^{T}$, we get  $u_{i}=-\alpha_{i}u_{n+1}$ for all $i\in I_{n}$.
 Hence, $D_{\tilde{\Phi}}$ is the set of  all  $\{g_{i}\}_{i=1}^{n+1}$ so that
\begin{eqnarray*}
g_{i}^{k} =\begin{cases}
\begin{array}{ccc}
 -\alpha\alpha_{i}\alpha_{k}-\alpha_{i}x_{k}& \;
{k\neq i}, \\
\alpha(1+\sum_{j\neq i}\alpha_{j}^{2})-\alpha_{i}x_{k}& \; {k=i} , \\
\end{array}
\end{cases}
\end{eqnarray*}
where $\alpha=\dfrac{1}{det S_{\tilde{\Phi}}}=\dfrac{1}{1+\sum_{i=1}^{n}\alpha_{i}^{2}}$, $g_{i}^{k}$ denotes $k^{th}$ coordinate of $g_{i}$,  $i\in I_{n}$, and
\begin{equation*}
g_{n+1}=\{\alpha\alpha_{i}+x_{i}\}_{i=1}^{n}.
\end{equation*}
This shows  that, every dual frame is associated to $n$-variable $\{x_{i}\}_{i\in I_{n}}$ and a dual frame is full spark frame except the cases that the sequence $u_{n+1}=\{x_{i}\}_{i\in I_{n}}$  belongs to the roots of the   polynomial constructed by
$ det [g_{i_{1}}\vert \quad ... \quad  \vert  g_{i_{n}}]=0$,
for some  index set $\{i_{j}\}_{j=1}^{n}\subset I_{n+1}$.
Multiplying these $n+1$ polynomials yield an $n$-variable polynomial  $P_{\tilde{\Phi}}(x_{1}, ...,x_{n})$. So,   every full spark dual frame  of $\tilde{\Phi}$ is  obtained   by a generic choice of $u_{n+1}$ out of the roots of the polynomial $P_{\tilde{\Phi}}$.  Hence, by a similar approach to the proof of Theorem \ref{densemetric}, and considering the bijection map $\xi:D_{\tilde{\Phi}}\rightarrow  \mathbb{R}^{n}$ defined as $\xi(G) = u_{n+1}$, where $G=\{S_{\tilde{\Phi}}^{-1}\tilde{\phi}_{i}+u_{i}\}_{i=1}^{n+1}$,
  the complement of all full spark dual frames  of $\Phi $  in $D_{\Phi}$  is equivalent  to  the set of roots of $P_{\tilde{\Phi}}$, and so, the set of full spark dual frames  of $\Phi$ is embedded into a generic set in $\mathbb{R}^{n}$ by $\xi$.
In general case, if $\Phi=\mathcal{T}\tilde{\Phi}$, for an invertible operator $\mathcal{T}$ on $\mathbb{R}^{n}$  by using   Lemma \ref{equlem}, the set of full spark dual frames  of $\Phi$, $FD_{\Phi}$,  is the same $(\mathcal{T}^{*})^{-1}FD_{\tilde{\Phi}}$,
 and this completes the  proof.
\end{proof}
\subsection{Examples}
\begin{ex}
 Suppose that $\phi=\{\phi_{i}\}_{i=1}^{3}$ is a full spark frame for $\mathbb{R}^{2}$.
Since $E(\mathcal{\phi})=1$  and equivalent frames do the same phase retrieval we assume that
   $\phi$ is equivalent to  $\{\delta_{1},\delta_{2},\alpha_{1}\delta_{1}+\alpha_{2}\delta_{2}\}$, in  which both $\alpha_{1}$ and $\alpha_{2}$ are non-zero. In this case,
\begin{equation*}
D_{\phi} = \left\{ \left[
 \begin{array}{ccc}

\alpha(1+\alpha_{2}^{2})-\alpha_{1}x\\
\\
-\alpha\alpha_{1}\alpha_{2}-\alpha_{1}y \\
\end{array} \right], \left[
 \begin{array}{ccc}

-\alpha\alpha_{1}\alpha_{2}-\alpha_{2}x\\\
\\
\alpha(1+\alpha_{1}^{2})-\alpha_{2}y \\
\end{array} \right], \left[
 \begin{array}{ccc}

\alpha\alpha_{1}+x\\
\\
\alpha\alpha_{2}+y
\end{array} \right] ; x,y\in \mathbb{R}\right\}
\end{equation*}
 where $\alpha=\dfrac{1}{1+\alpha_{1}^{2}+\alpha_{2}^{2}}$.
It is worthy of note that, all dual frames in this case are full spark and so  phase retrieval except dual frames obtained by $(x,y)\in \mathbb{R}^{2}$ on three  distinct lines with slopes of $0$, $\infty$ and $-\alpha_{1}/\alpha_{2}$ such as $x=-\alpha\alpha_{1}$, $y=-\alpha\alpha_{2}$ and $\alpha_{1}x+\alpha_{2}y=\alpha$.
 Considering
\begin{equation*}
(x,y)\in \mathbb{R}^{2}\setminus \{(x,y) : (x+\alpha\alpha_{1})(y+\alpha\alpha_{2})(\alpha_{1}x+\alpha_{2}y-\alpha)=0\}
\end{equation*}
 we get a generic choice of $\{u_{i}\}_{i\in I_{4}}$ in $\mathbb{R}^{2}$. The fact that all dual frames of $\phi$  are translations of this family by $\{S_{\phi}^{-1}\phi_{i}\}_{i\in I_{4}}$ shows that full spark dual frames are dense in $D_{\Phi}$ and full measure. As a special case,
considering $\alpha_{1}=\alpha_{2}=1$ we get the phase retrieval frame $\phi=\{\delta_{1},\delta_{2},\delta_{1}+\delta_{2}\}$
%
%
%
and we can see that dual frames of $\Phi$ do  phase retrieval for all $(x,y)\in \mathbb{R}^{2}$ except on the lines $x=\dfrac{-1}{3}$, $y=\dfrac{-1}{3}$ and $y=\dfrac{1}{3}-x$.
Put  $x=0$, $y=2/3$ we get a phase retrieval dual $\Psi$
 and
$x=1, y=\dfrac{-2}{3}$, satisfying in $y=\dfrac{1}{3}-x$, gives a non-phase retrieval dual frame $G$. See Figure 1.
\begin{center}
\includegraphics[scale=0.30]{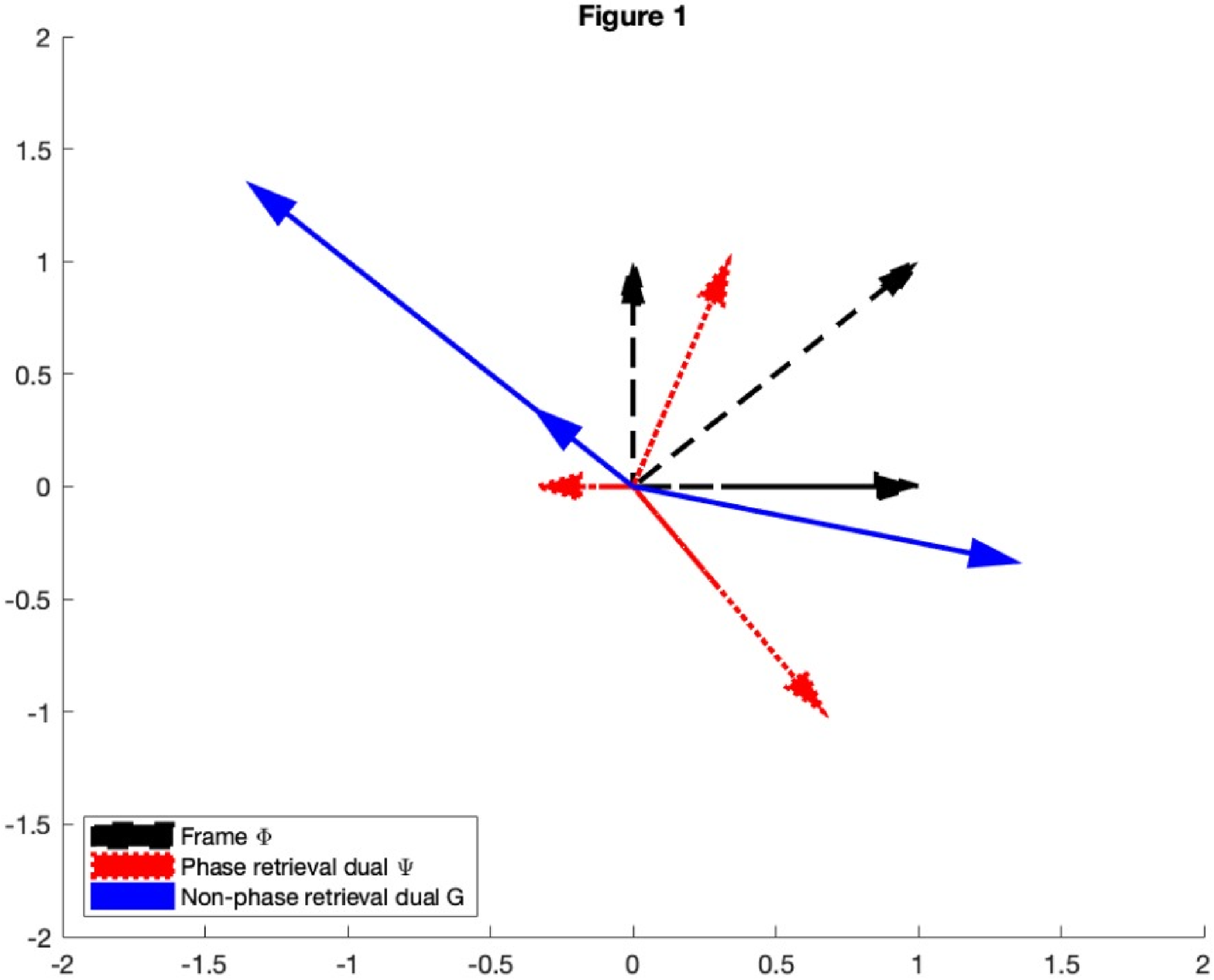}
\end{center}
\end{ex}
\begin{ex} Let $\phi=\{\phi_{i}\}_{i=1}^{4}$ be a full spark frame for $\mathbb{R}^{3}$.
In this case $\phi$ is equivalent to the frame
\begin{equation*}
\phi:= \{\delta_{1},\delta_{2}, \delta_{3}, \alpha_{1}\delta_{1}+\alpha_{2}\delta_{2}+\alpha_{3}\delta_{3}\quad   \alpha_{1},\alpha_{2},\alpha_{3}\neq 0\}.
\end{equation*}
%
The set of all dual frames of $\phi$ is in the form of $D_{\phi}=\{g_{1},g_{2},g_{3},g_{4}\}$ where
\begin{equation*}
g_{1}= \left[
 \begin{array}{ccc}

\alpha(1+\alpha_{2}^{2}+\alpha_{3}^{2})-\alpha_{1}x\\
\\
-\alpha\alpha_{1}\alpha_{2}-\alpha_{1}y \\
\\
-\alpha\alpha_{1}\alpha_{3}-\alpha_{1}z\\
\end{array} \right], g_{2}=\left[
 \begin{array}{ccc}

-\alpha\alpha_{1}\alpha_{2}-\alpha_{2}x\\\
\\
\alpha(1+\alpha_{1}^{2}+\alpha_{3}^{2})-\alpha_{2}y \\
\\
-\alpha\alpha_{2}\alpha_{3}-\alpha_{2}z\\
\end{array} \right],\end{equation*}
\begin{equation*}
g_{3}=\left[
 \begin{array}{ccc}

-\alpha\alpha_{1}\alpha_{3}-\alpha_{3}x\\
\\
-\alpha\alpha_{2}\alpha_{3}-\alpha_{3}y\\
\\
\alpha(1+\alpha_{1}^{2}+\alpha_{3}^{2})-\alpha_{3}z\\
\end{array} \right] , g_{4}=\left[
 \begin{array}{ccc}

\alpha\alpha_{1}+x\\
\\
\alpha\alpha_{2}+y\\
\\
\alpha\alpha_{3}+z
\end{array} \right]
\end{equation*}
where $\alpha=\dfrac{1}{1+\alpha_{1}^{2}+\alpha_{2}^{2}+\alpha_{3}^{2}}$ and $ x,y,z\in \mathbb{R}$ are arbitrary.  All dual frames of $\Phi$ are full spark except the cases $ det [g_{i}\vert  g_{j} \vert  g_{k}]=0$, for $i,j,k\in I_{4}$,
 in which $(x,y,z)$ is belong to the four distinct planes $x=-\alpha\alpha_{1}$, $y=-\alpha\alpha_{2}$, $z=-\alpha\alpha_{3}$ and $\alpha_{1}x+\alpha_{2}y+\alpha_{3}z=\alpha$. In a similar way, by choosing
\begin{equation*}
(x,y,z)\in \mathbb{R}^{3}\setminus \{(x,y,z) : (x+\alpha\alpha_{1})(y+\alpha\alpha_{2})(z+\alpha\alpha_{3})(\alpha_{1}x+\alpha_{2}y+\alpha_{3}z-\alpha)=0\}
\end{equation*}
 we can say that  full spark dual frames  are translations of the canonical dual by a generic choice of the family $\{u_{i}\}_{i\in I_{4}}$.
\end{ex}

\begin{ex}
Consider $\phi=\{\delta_{1},\delta_{2},\delta_{3},\sum_{i=1}^{3}\delta_{i}, \delta_{1}-\delta_{2}+\delta_{3}\}$ as a full spark frame for $\mathbb{R}^{3}$.
In this case, all  dual frames are  constructed by a generic choice of $\left[
 \begin{array}{ccc}

u_{4}\\
u_{5}  \\
\end{array} \right]\in \mathbb{R}^{6}$. In fact, by putting $u_{4}=[x_{1}\quad  y_{1} \quad z_{1}]^{T}$ and $u_{5}=[x_{2}\quad y_{2} \quad z_{2}]^{T}$,
 the set of all dual frames of $\phi$,  are given by
\begin{equation*}
g_{1}= \left[
 \begin{array}{ccc}

\dfrac{3}{5}-x_{1}-x_{2}\\
\\
-y_{1}-y_{2} \\
\\
\dfrac{-2}{5}-z_{1}-z_{2}\\
\end{array} \right], g_{2}=\left[
 \begin{array}{ccc}

x_{2}-x_{1}\\\
\\
\dfrac{1}{3}+y_{2}-y_{1} \\
\\
-z_{2}-z_{1}\\
\end{array} \right],\end{equation*}
\begin{equation*}
g_{3}=\left[
 \begin{array}{ccc}

\dfrac{-2}{5}-x_{1}-x_{2}\\
\\
-y_{1}-y_{2} \\
\\
\dfrac{3}{5}-z_{1}-z_{2}\\
\end{array} \right] , g_{4}=\left[
 \begin{array}{ccc}

\dfrac{1}{5}+x_{1}\\
\\
\dfrac{1}{3}+y_{1}\\
\\
\dfrac{1}{5}+z_{1}\\
\end{array} \right]
, g_{5}=\left[
 \begin{array}{ccc}

\dfrac{1}{5}+x_{1}\\
\\
\dfrac{-1}{3}+y_{2}\\
\\
\dfrac{1}{5}+z_{2}\\
\end{array} \right]
\end{equation*}
where $x_{1},x_{2},y_{1},y_{2},z_{1},z_{2}$ are obtained arbitrarily from $\mathbb{R}$. And all   cases in which a dual frame fails to do phase retrieval is associated to  the roots of a $6$-variable polynomial in $\mathbb{R}^{6}$ given by multiplying of the  polynomials as
$
 det [g_{i_{1}}\vert g_{i_{2}}\vert  g_{i_{3}}]  =0,
$
for all index set $\{i_{j}\}_{j=1}^{3}\subset I_{5}$.  As one case,  $det [g_{2}\vert  g_{4}  \vert  g_{5}]  =0$ deduces that
\begin{equation*}
z_{1}+x_{2}-\dfrac{3x_{1}}{5}-\dfrac{3z_{2}}{5}+3x_{2}z_{1}-3x_{1}z_{2}=0,
\end{equation*}
and the roots of this equation, which are associated to a family of non-phase retrieval dual frames, constitute a surface in $\mathbb{R}^{3}$.
\end{ex}

\section{Weak  Phase Retrieval}
The main result of this section is to    obtain  some  equivalent conditions on a family of vectors to do    phase  retrieval in terms of weak phase retrieval. This also derives a relationship between, phase, weak phase and norm  retrieval. First we present some properties of a family of vectors to do  weak phase retrieval.
\begin{prop}
Assume that $\Phi=\{\mathcal{\phi}_{i}\}_{i\in I_{m}}$ is a frame   in $\mathcal{H}_{n}$. Then  $\Phi$ does weak phase retrieval if and only if $P\Phi$ does weak phase retrieval, for every $2$-dimensional orthogonal projection  $P$ on $\mathcal{H}_{n}$.
\end{prop}
\begin{proof}
 In case $\Phi$ does weak phase retrieval, it is simple to see that $P\Phi$ does  weak phase retrieval, for every  orthogonal projection  $P$ on $\mathcal{H}_{n}$, as well, see also \cite{Akrami}. Conversely,
let $x,y\in  \mathcal{H}_{n}$ and $\vert\langle x,\phi_{i}\rangle\vert =\vert\langle y, \phi_{i}\rangle\vert$. Then by the assumption that $P\Phi$ does weak  retrieval for the projection $P$ of $\mathcal{H}_{n}$ onto the  closed subspace  $span\{x,y\}$, implies that $  x$ and  $y$ weakly have the same phase.
\end{proof}
It is shown that \cite{Akrami} a weak phase retrieval family in $\mathbb{R}^{n}$ spans the space that  means  such a family constitutes a frame for  $\mathbb{R}^{n}$. In the complex case $\mathbb{C}^{n}$ there is no result available.  In the following, we present   sufficient condition in this regard which is also useful for the main result of this section.
\begin{prop}\label{3prop}
Let $\{\phi_{i}\}_{i\in I_{m}}$ be a family of vectors in $\mathbb{C}^{n}$ so that  $U\Phi$ does weak phase retrieval for every unitary operator $U$  on $\mathbb{C}^{n}$. Then $\{\phi_{i}\}_{i\in I_{m}}$ is a frame for $\mathbb{C}^{n}$.
\end{prop}
\begin{proof}
Assume that  there exists an element $x\in \{\phi_{i}\}_{i\in I_{m}}^{\perp}$, and      get an ONB for $\mathbb{C}^{n}$ as $\{e_{i}\}_{i\in I_{n}}$ with $e_{1}=\dfrac{x}{\Vert x \Vert}$. Also,   take $0\neq y\in span\{\phi_{i}\}_{i\in I_{m}}$ so there exists $\{\beta_{i}\}_{i=2}^{n}\subset \mathbb{C}$, with some non-zero elements  such that $y=\sum_{i=2}^{n}\beta_{i}e_{i}$. Define
\begin{equation*}
\mu: \mathbb{C}^{n}\rightarrow \mathbb{C}^{n}; \quad \mu (e_{i})=\delta_{i}
\end{equation*}
where $\{\delta_{i}\}_{i\in I_{n}}$  is the standard ONB of $\mathbb{C}^{n}$. Clearly $\mu$ is a unitary operator and
\begin{equation*}
\mu(x+y)=\mu(x)+\mu(y)=\Vert x\Vert \delta_{1}+\sum_{i=2}^{n}\beta_{i}\delta_{i}, \quad \mu(-x+y)=-\Vert x\Vert \delta_{1}+\sum_{i=2}^{n}\beta_{i}\delta_{i}.
\end{equation*}
On the other hand,
\begin{eqnarray*}
\vert \langle \mu (x+y),\mu \phi_{i}\rangle\vert
&=& \vert \langle x+y ,\phi_{i}\rangle\vert \\
&=& \vert \langle  -x+y ,  \phi_{i}\rangle\vert\\
&=& \vert \langle \mu (-x+y),\mu \phi_{i}\rangle\vert,
\end{eqnarray*}
for all $i \in I_{m}$. The assumption that  $\mu \Phi$ does weak phase retrieval deduces that  $\mu (x+y)$ and $\mu (-x+y)$ weakly have the same phase
that is impossible, unless $x=0$. Therefore, $\Phi$  is a frame for $\mathbb{C}^{n}$.
\end{proof}

 \begin{thm}
Let $\Phi=\{\mathcal{\phi}_{i}\}_{i\in I_{m}}$ be a family of vectors in $\mathcal{H}_{n}$. Then the followings are equivalent;
\begin{itemize}
\item[(i)]
$\Phi$ does phase retrieval.
\item[(ii)] $U\Phi$ does phase retrieval for every  invertible  operator $U$ on $\mathcal{H}_{n}$.
\item[(iii)] $U\Phi$ does norm   retrieval for every  invertible  operator $U$ on $\mathcal{H}_{n}$.
\item[(iv)] $U\Phi$ does weak phase retrieval for every  invertible  operator $U$  on $\mathcal{H}_{n}$.
\end{itemize}
\end{thm}
\begin{proof}
The equivalency of $(i)$, $(ii)$, and $(iii)$ was proven in \cite{B-Cahill}.
Also, we clearly have $(i) \Rightarrow (ii)  \Rightarrow (iv)$.
So, it is sufficient to show that $(iv) \Rightarrow (i)$.
Assuming that $U\Phi$ does weak phase retrieval for all  invertible operators on $\mathcal{H}_{n}$, specially we imply that $\mathcal{\phi}$ does weak phase retrieval. Moreover, $\Phi$ is a frame for $\mathcal{H}_{n}$, i.e., $\Phi$ spans $\mathcal{H}_{n}$ by Proposition \ref{3prop}.
In the sequel, we are going to show that $\Phi$ yields phase retrieval. On the contrary, let  there exist non-zero elements $x$ and $y$ in $\mathcal{H}_{n}$ so that
\begin{equation}\label{eq41}
\vert \langle x,\phi_{i}\rangle\vert=\vert \langle y,\phi_{i}\rangle\vert, \quad  (i\in I_{m})
\end{equation}
 but $y \neq c x$ for any $\vert c\vert = 1$. Since $\Phi$ does weak phase retrieval, there exists  some $\vert \theta\vert=1$ so that $\theta phase(x_{j})=phase (y_{j})$
 for all $j\in I_{n}$. The assumption that $y\neq \theta x$ implies $\vert x_{j}\vert \neq  \vert y_{j}\vert $ for some $j\in I_{n}$. Since $\Phi$ is a frame, the equality in $(\ref{eq41})$ is non-zero for some $i$ then $y= cx$ immediately implies that  $x$ and $y$ are linearly independent and we can  consider a basis for $\mathcal{H}_{n}$ containing $x$ and $y$ as $\{x,y, e_{3}, ..., e_{n}\}$. We face the following   cases

Case $1$. If $x$ and $y$ are disjointly supported, then there exist no   non-zero common coordinates.  Without loss  of generality we  let   $x_{1}, y_{2} \neq 0$,  and so $x_{2}= y_{1} = 0$. Define $U:\mathcal{H}_{n}\rightarrow \mathcal{H}_{n}$ so that $Ux= x- y$, $Uy=  x+ y$ and $Ue_{k}=e_{k}$, $3\leq k\leq n$. Then, $U$ is a bounded invertible operator on $\mathcal{H}_{n}$ and
\begin{equation*}
\vert \langle Ux,(U^{-1})^{*}\phi_{i}\rangle\vert=\vert \langle Uy,(U^{-1})^{*}\phi_{i}\rangle\vert, \quad (i \in I_{m}),
\end{equation*}
 by $(\ref{eq41})$ and the invertibility of $U$. However, $Ux$ and $Uy$ have not weakly the same phase  due to $phase(U x)_{1}=phase (Uy)_{1}$ and  $ phase(U x)_{2}= e^{i\pi}phase (Uy)_{2}$.
This contradicts   the assumption that $(U^{-1})^{*}\Phi$ does  weak phase retrieval.

 Case $2$. Let $x$ and $y$ have one  non-zero  coordinate  in common in $i$th index. If $y_{l}= 0=x_{l}$ for all $l\neq i$ then by $(\ref{eq41})$ and using the assumption $\vert x_{i}\vert=\vert y_{i}\vert$, i.e., $y=\theta x$ that is a contradiction. Thus,
 there exists  $l\neq i$ so that $y_{l}\neq 0$ or $x_{l}\neq 0$. Without loss  of  generality we   let
 $y_{l}\neq 0$ and
 $\dfrac{\vert x_{l}\vert}{\vert y_{l}\vert}< \dfrac{\vert x_{i}\vert}{\vert y_{i}\vert}$.
In fact, if in all  common non-zero elements $\dfrac{\vert x_{i}\vert}{\vert y_{i}\vert}=c$ then $y=\dfrac{\theta}{c}x$, which contradicts by the assumption.

So, we get $\epsilon>0$ so that $\dfrac{\vert x_{l}\vert}{\vert y_{l}\vert}<\epsilon< \dfrac{\vert x_{i}\vert}{\vert y_{i}\vert}$.  Define
$U:\mathcal{H}_{n}\rightarrow \mathcal{H}_{n}$ so that $Ux=\theta x-\epsilon y$, $Uy=\theta x+\epsilon y$ and $Ue_{k}=e_{k}$, $3\leq k\leq n$.
 Moreover,
\begin{eqnarray*}
(Ux)_{j}=(\vert x_{j}\vert - \epsilon\vert y_{j}\vert)\alpha_{j}\theta, \quad (Uy)_{j}=(\vert x_{j}\vert +\epsilon \vert y_{j}\vert)\alpha_{j}\theta. \quad (j\in I_{n})
\end{eqnarray*}
where $\alpha_{j}=phase(x_{j})$. Hence, we obtain $phase (Ux)_{i}=phase (Uy)_{i}$, however $phase (Ux)_{l} =e^{i\pi}phase  (Uy)_{l}$.
 That leads to  a contradiction similar to the previous case, as required.


\end{proof}
The following result gives a sufficient condition on a   family of frames  so that their  canonical dual also yields weak phase retrieval.
\begin{prop}
Let $\Phi=\{\mathcal{\phi}_{i}\}_{i\in I_{m}}$ be a frame   in $\mathcal{H}_{n}$ with diagonal frame operator. Then $\Phi$ does weak phase retrieval if and only if its canonical dual does so.
\end{prop}
\begin{proof}
Suppose that $\alpha_{1}, ..., \alpha_{n}$ be the diagonal elements of $S_{\Phi}$, respectively.
Obviously, $\alpha_{i}> 0$, for all $i\in I_{n}$. Let $\Phi$ does weak phase retrieval, take $x,y\in \mathcal{H}_{n}$ such that $\vert\langle x, S_{\Phi}^{-1}\phi_{i}\rangle\vert=\vert\langle y, S_{\Phi}^{-1}\phi_{i}\rangle\vert$. Then,  we get
$S_{\Phi}^{-1}x=(x_{1}/ \alpha_{1},...,x_{n}/ \alpha_{n})$ and $S_{\Phi}^{-1}y=(y_{1}/\alpha_{1},..., y_{n}/\alpha_{n})$ weakly have the same phase. Consequently, $x$ and $y$ weakly have the same phase, as well. The converse is implied by a similar explanation.
\end{proof}

\textbf{Funding.} The authors have not disclosed any funding.\\

\textbf{Data Availability Statement.}  Data sharing not applicable to this article as no
datasets were generated during the preparation of this paper.\\

\textbf{Conflict-of-interest.} This work does not have any conflicts of interest.

\bibliographystyle{amsplain}



\end{document}